\documentclass[11pt]{article}
\usepackage[T1]{fontenc}
\usepackage{enumerate}
\usepackage{amsmath}
\usepackage{amsthm}
\usepackage{amsfonts}
\usepackage{amssymb}
\usepackage[numbers]{natbib}
\usepackage{xcolor}
\usepackage{graphicx}
\usepackage{tikz}
\usepackage{fullpage}
\usepackage{xcolor}
\usepackage{hyperref}
\usepackage{bbm}

\newtheorem{question}{Question}
\numberwithin{question}{section}
\newtheorem{conjecture}[question]{Conjecture}
\newtheorem{problem}[question]{Problem}

\newtheorem{theorem}[question]{Theorem}
\newtheorem{proposition}[question]{Proposition}
\newtheorem{lemma}[question]{Lemma}
\newtheorem{claim}[question]{Claim}

\newtheorem{definition}[question]{Definition}

\definecolor{caPurple}{RGB}{120,94,240}

\newcommand{\sat}{\mathrm{sat}}

\newtheorem{remark}[question]{Remark}
\numberwithin{equation}{section}

\newcommand*\samethanks[1][\value{footnote}]{\footnotemark[#1]}
\title{Induced poset saturation in the hypergrid}

\author{R. Altar \c{C}i\c{c}eksiz\thanks{Institutionen för matematik och matematisk statistik, Umeå Universitet, 901 817 Ume{\aa}, Sweden.\\	Emails: $\{$\texttt{altar.ciceksiz}, \texttt{victor.falgas-ravry}, \texttt{sabrina.lato}, \texttt{maryam.sharifzadeh}$\}$\texttt{@umu.se}. Research supported by Kempe Stiftelse grant JCSMK22-0160 and Vetenskapsr{\aa}det grant VR 2021-03687.} \and Victor Falgas--Ravry\samethanks[1] \and Sabrina Lato\samethanks[1] \and Maryam Sharifzadeh\samethanks[1]}

\begin{document}
	\maketitle
	\begin{abstract}
		Set $[n]=\{1, 2, \ldots , n\}$. The hypergrid $[t]^n$ is the collection of functions $f: \ [n]\rightarrow [t]$. We equip it with the natural partial order by letting $f\leq g$ whenever $f(x)\leq g(x)$ holds for all $x\in [n]$. Given a poset $P$ which can be embedded as an induced subposet of $[t]^n$, the induced poset saturation function $\mathrm{sat}^{\star}([t]^n, P)$ denotes the minimum size of a subset of $[t]^n$ that is both induced $P$-free and induced $P$-saturated.

		We show that for all $t\geq 2$, $\mathrm{sat}^{\star}([t]^n, P)$ satisfies a dichotomy: for every poset $P$, either there exists a constant $C_P$ such that $\mathrm{sat}^{\star}([t]^n, P)=C_P$ for all $n$ sufficiently large, or $\mathrm{sat}^{\star}([t]^n, P)=\Omega(\sqrt{n})$. We also show chains fall in the former part of the dichotomy, while posets with the unique twin cover property fall in the latter part. These contributions generalize a number of results obtained by various authors in the hypercube ($t=2$) setting; the transition to the hypergrid setting provides novel challenges, however, and requires some new ideas. 
	\end{abstract}

\section{Introduction}

For $n,t\in \mathbb{Z}_{>0}$, let $[t]^n$ denote the set of functions from $[n]:=\{1,2,\ldots,n\}$ to $[t]:=\{1,2,\ldots,t\}$, ordered coordinatewise: for $f,g\in [t]^n$ we write $f\leq g$ if $f(i)\leq g(i)$ for all $i\in [n]$. When $t=2$, this is the Boolean lattice $[2]^{n}$, one of the central objects in extremal set theory. For general $t\geq 3$, the hypergrid $[t]^n$ is a natural extension of the hypercube and has already appeared in a range of combinatorial settings, including antichain enumeration, chain decompositions, integer partitions, Ramsey theory, and questions related to discrete distributions; see for example~\cite{Anderson67,CarrollCooperTetali09,FRT25,Macmahon2001,MattnerRoos08,MoshkovitzShapira14,ParkSarantisTetali25,PohoataZakharov24,Tomon25,Tsai19}.

Many arguments that work in the hypercube do not extend directly to the hypergrid. The hypergrid remains a graded poset, but unlike the hypercube it is not biregular between consecutive levels, and this loss of symmetry often requires the use of new ideas when trying to generalize hypercube results. This phenomenon has already been observed in several extremal problems on $[t]^n$, and provides part of the motivation for the present work.

We study induced poset saturation in the hypergrid. Given a poset $P$, an \emph{induced copy} of $P$ in $[t]^n$ is an injective map $\phi:P\to [t]^n$ such that for all $x,y\in P$, $phi(x)\leq \phi(y)$ if and only if $x\leq_P y$.
A family $\mathcal{F}\subseteq [t]^n$ is \emph{induced $P$-free} if it contains no induced copy of $P$. It is \emph{induced $P$-saturated} if it is induced $P$-free, but for every $f\in [t]^n\setminus \mathcal{F}$ the family $\mathcal{F}\cup\{f\}$ contains an induced copy of $P$. Whenever $P$ embeds into $[t]^n$, we write $\sat^{\star}([t]^n,P)$ for the minimum size of an induced $P$-saturated family in $[t]^n$. Our aim is to understand the possible asymptotic behaviour of $\sat^{\star}([t]^n,P)$ for fixed $t$ and $n\to \infty$.

\begin{problem}[Induced saturation in the hypergrid]\label{problem: induced saturation in the hypergrid}
	Determine the possible range of behaviours of $\sat^{\star}([t]^n,P)$.
\end{problem}

The starting point for this question is the corresponding theory in the Boolean lattice. Saturation problems in posets were introduced by Gerbner, Keszegh, Lemons, Palmer, P\'alv\"olgyi and Patk\'os~\cite{GKLPPP}, extending the classical graph saturation problem of Erd\H{o}s, Hajnal and Moon. In the non-induced setting, the behaviour is remarkably tame: for every finite poset $P$, one has $\sat([2]^n,P)=O_P(1)$ in the Boolean lattice~\cite{GKLPPP,KLMPP}.

The induced setting is much subtler. Ferrara, Kay, Kramer, Martin, Reiniger, Smith and Sullivan~\cite{FKKMRSS} initiated the systematic study of induced poset saturation in the Boolean lattice. They established exact values and bounds for several small posets, and introduced the class of posets with the \emph{unique cover twin property} (UCTP), which we now define. An element $x\in P$ is said to \emph{cover} an element $y\in P\setminus\{y\}$ if $y\leq_P x$ and there is no element $z\in P$ distinct from $x,y$ such that $y\leq_P z\leq_P x$. A poset $P$ then has the UCTP property if whenever $x$ covers $y$ there exists $y' \in P\setminus\{x,y\}$ such that $x$ covers $y'$. Ferrara et al.~\cite{FKKMRSS} established that every UCTP poset $P$ on at least two elements satisfies $\sat^{\star}([2]^n,P)\geq \log_2 n$.

Keszegh, Lemons, Martin, P\'alv\"olgyi and Patk\'os~\cite{KLMPP} then proved a first general dichotomy: for every finite poset $P$, either $\sat^{\star}([2]^n,P)$ is bounded by a constant depending only on $P$, or $\sat^{\star}([2]^n,P)\geq \log_2 n$. This dichotomy was then improved by Freschi, Piga, Sharifzadeh and Treglown~\cite{FPST}, who showed that every unbounded induced saturation function is in fact at least of order $\sqrt{n}$. A central conjecture in the area, first stated by Keszegh et al.~\cite{KLMPP}, is that this dichotomy should be even sharper.
\begin{conjecture}[\cite{KLMPP}]
	\label{dichotomyconj}
	For every finite poset $P$, there exists a constant $C_P$ such that either $\sat^{\star}([2]^n,P)\leq C_P$ for all $n$, or $\sat^{\star}([2]^n,P)\geq n+1$ for all sufficiently large $n$.
\end{conjecture}
\noindent This would be best possible, since Ferrara et al.~\cite{FKKMRSS} proved that $\sat^{\star}([2]^n,\mathcal{V})=n+1$ for the fork poset $\mathcal{V}$.

There has also been substantial progress on upper bounds for induced saturation functions. Bastide, Groenland, Ivan and Johnston~\cite{BGIJ} proved that $\sat^{\star}([2]^n,P)$ is polynomial in $n$ for every fixed poset $P$. More precisely, they showed that
\[
\sat^{\star}(n,P)=O\bigl(n^{w^{\star}(P)}\bigr),
\]
where $w^{\star}(P)$ denotes the \emph{cube width} of $P$, an explicitly computable poset parameter introduced in~\cite{BGIJ}. They further conjectured that $w^{\star}(P)\leq \vert P\vert$. This was recently proved independently by Bastide, Hodor, La and Trotter~\cite{BHLTCubeWidth2} and by Ivan, Fl\'idr and Jaffe~\cite{IJFCubeWidth}. In particular, every induced saturation function in the Boolean lattice is bounded above by a polynomial whose degree is at most $\vert P\vert$.

Alongside these general developments, a substantial literature has emerged on the induced saturation problem for specific target posets. For antichains, Ferrara et al.~\cite{FKKMRSS} determined $\sat^{\star}([2]^n,\mathcal{A}_k)$ for $k=2,3,4$ and conjectured that $\sat^{\star}([2]^n,\mathcal{A}_k)\sim (k-1)n$ for fixed $k\geq 3$. Martin, Smith and Walker~\cite{MSW} obtained improved lower bounds, {\DJ}ankovi\'c and Ivan~\cite{DI} determined the exact values for $k=5,6$, and Bastide, Groenland, Jacob and Johnston~\cite{BGJJ}, extending work of Lehman and Ron~\cite{LR}, established asymptotic results and exact values for infinitely many values of $k$.

For the diamond poset $\mathcal{D}_2$ (which corresponds exactly to the $2$-dimensional Boolean lattice $[2]^2$), Martin, Smith and Walker~\cite{MSW} proved the first nontrivial lower bound. A constant fact improvement was then obtained by Ivan~\cite{Ivan} in 2022, before Ivan and Jaffe~\cite{IvanJaffeDiamond} very recently proved the linear bound $\sat^{\star}([2]^n,\mathcal{D}_2)\geq (n+1)/5$, thereby determining the correct order of magnitude. Even more recently, the same two authors~\cite{IvanJaffeDiamond2} showed $\sat^{\star}([2]^n,\mathcal{D}_2)=n+1$, settling one of the prominent conjectures in the field.

Ivan~\cite{Ivan2} proved a linear lower bound for the butterfly poset and polynomial upper bounds for the complete balanced bipartite poset $K_{k,k}$. Liu~\cite{Liu} later showed that $\sat^{\star}([2]^n,K_{s,t})=O(n)$ for all fixed $s,t\geq 2$, disproving a previous conjecture of superlinear growth in some cases. For the poset $\mathcal{N}$, Ivan~\cite{Ivan2} gave a first non-trivial lower bound on the induced saturation function, which Ivan and Wang~\cite{ivanwang25} recently improved to a linear bound. There has also been recent progress on structural operations preserving unboundedness or linear growth, for example in the work of Ivan and Jaffe on gluing and layered constructions~\cite{IvanJaffe1,IvanJaffe2}.

Against this background, it is natural to ask for analogous results in the hypergrid. This is not merely a formal generalization of the Boolean lattice. As noted above, extremal phenomena on $[t]^n$ have already been studied extensively in other contexts, and the passage from $[2]^{n}$ to $[t]^n$ typically introduces genuine new difficulties. In a different asymptotic regime, where the dimension $n=2$ is fixed and the side length $t$ tends to infinity, Gerbner, Nagy, Patk\'os and Vizer~\cite{GerbnerNagyPatkosVizer22} studied induced poset saturation problems in the grid $[t]^2$ and proved a dichotomy result in that setting.

In this paper, we instead fix $t$ and let $n\to\infty$. Our main result shows that the hypergrid exhibits the same coarse dichotomy as the Boolean lattice: for every fixed poset $P$, either $\sat^{\star}([t]^n,P)$ is eventually constant, or it is at least of order $\sqrt{n}$. We also identify classes of posets lying on either side of this dichotomy. In particular, chains have bounded induced saturation number in the hypergrid, while UCTP posets with at least two elements have unbounded induced saturation number. Thus several of the main qualitative features of induced saturation in the Boolean lattice persist in the hypergrid, even though the proofs require new ideas adapted to the more irregular structure of $[t]^n$.

The remainder of the paper is organized as follows. In Section~\ref{section: results} we state our main results. Their proofs are given in Section~\ref{section: proofs}. We conclude in Section~\ref{section: concluding remarks} with further remarks and a number of open problems.
	\subsection{Contributions of this paper}\label{section: results}
	\noindent We establish a strong dichotomy result for induced poset saturation in the hypergrid:
	\begin{theorem}[Strong dichotomy]\label{theorem: strong dichotomy}
		For any poset $P$ and all $t\in \mathbb{Z}_{\geq 2}$, either $\mathrm{sat}^{\star}([t]^n, P)\geq \Omega(\sqrt{n})$ or $\mathrm{sat}^{\star}([t]^n, P)=O(1)$.    
	\end{theorem}
	\noindent Theorem~\ref{theorem: strong dichotomy} generalizes a result of Freschi, Piga, Sharifzadeh and Treglown in the hypercube setting~\cite{FPST}. Our approach, relying on the notion of separation, is indebted to that of~\cite{FPST}; however, their argument does not carry over to the hypergrid setting, and the proof of Theorem~\ref{theorem: strong dichotomy} requires several new ideas, which are the main contribution of this paper.
	
	
	It is by no means self-evident whether or not $\mathrm{sat}^{\star}([t]^n, P)$ must be monotonically non-decreasing in $n$ in general. In our second result, we show that in the case where $\mathrm{sat}^{\star}([t]^n, P)=O(1)$, the induced saturation function must eventually be constant.
	\begin{theorem}\label{theorem: monotonicity}
		Suppose $\mathrm{sat}^{\star}([t]^n, P)=O(1)$ for some $t\in \mathbb{Z}_{\geq 2}$. Then there exists $N\in \mathbb{Z}_{>0}$ such that $n\mapsto \mathrm{sat}^{\star}([t]^n, P)$ is a constant function when restricted to $\mathbb{Z}_{>N}$.
	\end{theorem}
	\noindent Further, we are able to show that $\mathrm{sat}^{\star}([t]^n, P)$ grows polynomially for the large and well-studied family of UCTP posets, generalising in a strong form a result of Ferrara, Kay, Kramer, Martin, Reiniger, Smith and Sullivan~\cite{FKKMRSS} from the hypercube setting.
	\begin{theorem}[Large induced saturation function for UCTP posets]\label{theorem: strong dichotomy UCTP}
		For any UCTP poset $P$ on at least $2$ elements, $\mathrm{sat}^{\star}([t]^n, P)\geq \Omega(\sqrt{n})$.    
	\end{theorem}
	\noindent We also derive some elementary bounds on the induced saturation function of chains and antichains, showing these fall into different parts of the dichotomy we have established.
	\begin{proposition}
		\label{proposition: small examples}
		For all $k,t\in \mathbb{Z}_{\geq 2}$, we have $\mathrm{sat}^{\star}([t]^n, C_k)\leq 2^{k-2}$ and $\mathrm{sat}^{\star}([t]^n, A_k)=\theta(n)$.
	\end{proposition}
	\noindent Finally we obtain some polynomial upper bounds on $\mathrm{sat}^{\star}([t]^n, P)$ for a wide family of posets. 
       Recall that given two posets $P, Q$, we may define the \emph{poset sum} $P\star Q$ to be the poset obtained by taking a disjoint union of the posets $P$ and $Q$ and adding relations so that every element of $P$ lies below every element of $Q$. Note that $\star$ is associative so that we can use $A\star B\star C$ to denote $(A\star B)\star C=A\star (B\star C)$.
	\begin{theorem}\label{theorem: polynomial upper bound}
		Let $k\in \mathbb{Z}_{\geq 1}$ be fixed, and let $P= P_1\star A_k \star P_2$ for some arbitrary (possibly empty) posets $P_1$ and $P_2$. Then  for all $t\in \mathbb{Z}_{\geq 2}$, and all $n$ sufficiently large, $\mathrm{sat}([t]^n, P) \leq \left(nt^2\right)^{c}$ for some constant $c=c(P)$.   
	\end{theorem}
	\noindent We believe, however, that all these results fall far short of the actual truth, and make the following bold conjectures.
	\begin{conjecture}\label{conjecture: dichotomy}
		For any poset $P$ and $t\in \mathbb{Z}_{\geq 2}$ fixed,  $\mathrm{sat}^{\star}([t]^n, P)$ is either $O(1)$ or of the form $cn+o(n)$ for some constant $c=c(t, P)\in \mathbb{R}_{>0}$.  
	\end{conjecture}
	\begin{conjecture}\label{conjecture: monotonicity}
		For any poset $P$, the function $(n,t)\mapsto\mathrm{sat}^{\star}([t]^n, P)$ is monotone non-decreasing in both $n$ and $t$ in the range of parameters $(n,t)$ for which it is defined.
	\end{conjecture}

	\section{Proofs}\label{section: proofs}
	We shall make use of a coordinate-lifting operator $L_i$ and a coordinate-dropping operator $D_i$. Given $f\in [t]^n$ and a coordinate $i\in [n]$, we let $L_i(f)\in [t]^{n+1}$ denote the `lifted' function 
	\begin{align*}
		L_i(f): \  x \mapsto \begin{cases} f(x) & \textrm{if }x\in [n],\\
			f(i) & \textrm{if }x=n+1,
		\end{cases}
	\end{align*}
	that is obtained from $f$ by copying coordinate $i$. Similarly, given $f\in [t]^n$, we let $D_i(f)\in[t]^{n-1}$ denote the function obtained by dropping coordinate $i$,
	\begin{align*}
		D_i(f): \  x \mapsto \begin{cases} f(x) & \textrm{if } 1 \leq x < i\\
			f(x+1) & \textrm{if } i \leq x \leq n-1.
		\end{cases}
	\end{align*}
	
	\noindent We write $L_i(\mathcal{F}):=\{ L_i(f): \ f\in \mathcal{F}\}$ and $D_i(\mathcal{F}):=\{D_i(f): \ f\in \mathcal{F}\}$. Note that $L_i$ is an injection and preserves both comparability and incomparability, while $D_i$ preserves all comparability relations, but does not necessarily preserve incomparability in general.

	Throughout this section, given $f\in [t]^n$ and $I\subseteq [n]$ we write $f_{\vert I}$ for the restriction of $f$ to $I$; further if $I\subseteq [m]$ then for $g\in [t]^m$ we write $f \leq_I g$ and $f=_I g$ as a shorthand for the pointwise inequality and equality $f_{\vert I}\leq g_{\vert I}$ and $f_{\vert I}=g_{\vert I}$ respectively.
	\subsection{Warm-up: proof of a weak dichotomy theorem}
	\begin{lemma}\label{lemma: repeated coordinates imply can lift}
		Suppose $\mathcal{F}$ is an induced $P$-free saturated family in $[t]^n$ such that there exist $i,i'$ with $1\leq i<i' \leq n$ and $f(i)=f(i')$ for all $f\in \mathcal{F}$. Then $L_i(\mathcal{F})$ is an induced $P$-free saturated family in $[t]^{n+1}$.  
	\end{lemma}
	\begin{proof}
		By the properties of the lifting operator, $L_i(\mathcal{F})$ is induced $P$-free, so we only need to establish the saturation property. Consider  $g\in [t]^{n+1}\setminus L_i(\mathcal{F})$. We claim there is a copy of $P$ in $L_i(\mathcal{F})\cup\{g\}$.

		If $g=L_i(g')$ for some $g'\notin \mathcal{F}$, then there is a copy of $P$ in $\mathcal{F}\cup \{g'\}$, which is lifted under $L_i$ to a copy of $P$ in $L_i(\mathcal{F})\cup \{g\}$. As functions in $L_i(\mathcal{F})$ are constant functions when restricted to $\{i,i', n+1\}$, a symmetric argument applies if any two of $g(i), g(i'), g(n+1)$ take the same value (since $g$ may be viewed as a lifted copy of some $g'\notin \mathcal{F}$ in all such cases). 
		
		
		By symmetry (permuting the coordinates $i,i',n +1$ as necessary) we may thus assume without loss of generality that $g(i)<g(n+1)<g(i')$. Now consider $D_{n+1}(g)$. This is not an element of $\mathcal{F}$, since it is not constant on $\{i,i'\}$, hence there is a set $Q\subseteq \mathcal{F}$ such that $Q\cup\{D_{n+1}(g)\} $ induces a copy of $P$ in $[t]^n$. It is then easily seen that $L_i(Q)\cup\{g\}$ induces a copy of $P$ in $[t]^{n+1}$: since functions in $Q$ are constant over $\{i,i'\}$ and since $g(i)<g(n+1)<g(i')$, for any $q\in Q$  the comparability of $q_{\vert \{i,i'\}}$ and $D_{n+1}(g)_{\vert \{i,i'\}} $ is the same as the comparability of $(L_i(q))_{\vert \{i,i',n+1\}}$ and $g_{\vert \{i,i',n+1\}} $.  
	\end{proof}

\noindent As a corollary, we obtain a weak dichotomy result for induced poset saturation in the hypergrid:
\begin{theorem}[Weak dichotomy]\label{theorem: weak dichotomy}
	For any poset $P$, $\mathrm{sat}^{\star}([t]^n, P)$ is either $O(1)$ or at least $\log_t (n)$.    
\end{theorem}

\begin{proof}
	Suppose for some $N \in \mathbb{Z}_{\geq 0}$ there exists an induced $P$-free saturated family $\mathcal{F}\subseteq [t]^N$ such that for some $i<i'$ we have $f(i)=f(i')$ for all $f\in \mathcal{F}$. Then for any $n \geq N,$ we may apply Lemma~\ref{lemma: repeated coordinates imply can lift} $n-N$ times to see that $(L_i)^{(n-N)} (\mathcal{F})= L_i\circ L_i \circ \cdots \circ L_i (\mathcal{F})$ is an induced $P$-free-saturated family in $[t]^{n}$ of size $\vert \mathcal{F}\vert$. This implies that for all $n\geq N$, $\mathrm{sat}^{\star}([t]^{n}, P)\leq \vert \mathcal{F}\vert =O(1)$.
	
	On the other hand, if this does not happen then for every $n\in \mathbb{Z}_{>0}$, every induced $P$-free saturated family $\mathcal{F}\subseteq [t]^n$, and every $1\leq i<i'\leq n$ there exists $f\in \mathcal{F}$ with $f(i)\neq f(i')$. Set $m:=\vert \mathcal{F}\vert$ and enumerate the members of $\mathcal{F}$ as $f_1, f_2, \ldots, f_m$. By the pigeon-hole principle, there exists $I_1\subseteq [n]$ of size at least $n/t$ 
	such that $(f_1)_{\vert I_1}$ is a constant function. Iterating this argument, for all $1\leq k< m$ we have a subset $I_{k+1}\subseteq I_k$ of size at least $\vert I_k\vert/t \geq n/t^k$ such that $(f_{k+1})_{\vert I_k}$ is a constant function.

	In particular all functions in $\mathcal{F}$ are constant on $I_m$. 
	By our assumption, it follows that $1\geq \vert I_m\vert \geq n/t^m$, and thus (since $\mathcal{F}$ was an arbitrarily chosen induced $P$-free saturated family) that $\mathrm{sat}^{\star}([n]^t, P)\geq \log_t(n)$.
\end{proof}
		\subsection{Separation and proof of Theorem~\ref{theorem: strong dichotomy UCTP}}
		\noindent Recall that the coordinate-dropping operator $D_i$ does not preserve incomparability in general; analysing whether or not $D_i$ does preserve incomparability between members of a induced $P$-free saturating family plays an important role in our arguments, leading us to the following definition.
		\begin{definition}
			We say coordinate $i\in [n]$ is \emph{separating} for $\mathcal{F}\subseteq [t]^n$ if there exists distinct $f,f'\in \mathcal{F}$ such that $D_i(f) \leq D_i(f')$ and $f(i)>f'(i)$.     
		\end{definition}
		\noindent Adapting an argument of Freschi, Piga, Sharifzadeh and Treglown~\cite[Lemma 2.2]{FPST} in the hypercube case to the hypergrid setting, we show that families in $[t]^n$ for which every coordinate is separating must have size at least $\Omega(\sqrt{n})$.
		To do this we shall need the following extremal bound on the number of edges in a digraph with no transitively oriented cycle.
		For integer $k\geq 3$, let $\overrightarrow{TC_k}$ denote the transitively oriented cycle  on $k$ vertices, i.e.\ the directed graph on $[k]$ with directed edges $\overrightarrow{1k}$ and $\overrightarrow{i(i+1)}$ for $i\in [k]$.
		\begin{proposition}[Theorem 1.7 in~\cite{FPST}]\label{prop: TCk free digraph}
			Let $\overrightarrow{G}$ be a digraph on $m$ vertices which is $\overrightarrow{TC_k}$-free for all integers $k\geq 3$. Then $\overrightarrow{G}$ has at most $\max\{2(m-1), \lfloor\frac{m^2}{4}\rfloor\}$ edges.
		\end{proposition}
		\begin{lemma}\label{lemma: separating implies large}
			Let $\mathcal{F}\subseteq[t]^n$ be a family for which every coordinate $i\in [n]$ is separating. Then $\vert \mathcal{F}\vert =\Omega(\sqrt{n})$.
		\end{lemma}
		\begin{proof}
			We define a digraph $\overrightarrow{G}$ on the vertex set $\mathcal{F}$ by choosing for each $i\in [n]$ a pair of functions $f,f'\in \mathcal{F}$ such that $D_i(f)\leq D_i(f')$ and $f(i)>f'(i)$ and adding the directed edge $\overrightarrow{f'f}$ to $E(\overrightarrow{G})$. 
			%
			Once this is done, we prove an analogue of~\cite[Claim 2.4]{FPST}, so as to be able to apply Proposition~\ref{prop: TCk free digraph}. 
			\begin{claim}\label{claim: generalised claim 2.4}
				For every $k\geq 3$, $\overrightarrow{G}$ is $\overrightarrow{TC_k}$-free.  
			\end{claim}
			\begin{proof}
				Suppose that for some $k\geq 3$ our directed graph $\overrightarrow{G}$ contained edges $\overrightarrow{f_if_{i+1}}$ for $i\in [k]$ and distinct functions $f_1,f_2,\ldots,f_k$. By construction of $\overrightarrow{G}$, this means that there exist $k-1$ distinct coordinates $i_1,\ldots,i_{k-1}$ such that  $D_{i_j}(f_j)\geq D_{i_j}(f_{j+1})$, and $f_j(i_j)<f_{j+1}(i_j)$. In particular for all indices $i\in [n]\setminus\{i_1, \ldots, i_{k-1}\}$, we have $f_{k}(i)\leq f_1(i)$.
				Since we cannot find a coordinate $i_k \in [n]\setminus\{i_1, \ldots, i_{k-1}\}$ with $f_{k}(i_k) > f_1(i_k)$, we therefore have that the directed edge $\overrightarrow{f_1f_k}$ does not belong to $\overrightarrow{G}$. Thus $\overrightarrow{G}$ is $\overrightarrow{TC_k}$-free as claimed. 
			\end{proof}
			\noindent With Claim~\ref{claim: generalised claim 2.4} in hand, Lemma~\ref{lemma: separating implies large} follows directly 
			from Proposition~\ref{prop: TCk free digraph}, noting the $\overrightarrow{TC_k}$-free digraph $\overrightarrow{G}$ has $m=\vert\mathcal{F}\vert$ vertices and $n$ edges, which yields $n\leq \max\{2(\vert \mathcal{F}\vert-1), \lfloor\frac{\vert{\mathcal{F}}\vert^2}{4}\rfloor\}$.
		\end{proof}
		%
		%
			\noindent We now show that UCTP posets on at least three elements have induced $P$-free saturating families for which every coordinate is separating;
			combined with Lemma~\ref{lemma: separating implies large}, this easily implies Theorem~\ref{theorem: strong dichotomy UCTP}.
			\begin{lemma}\label{lemma: UCTP implies separating}
				Let $P$ be a UCTP poset on at least three elements. If $\mathcal{F}$ is an induced $P$-free saturated family in $[t]^n$, then for every coordinate $i\in [n]$, $i$ is separating for $\mathcal{F}$.
			\end{lemma}
			\begin{proof}
				By symmetry among the coordinates, it is sufficient to prove that $n$ is separating for $\mathcal{F}$. Suppose for a contradiction this does not hold.

				\noindent \textbf{Case 1: $f(n)=1$ for every $f\in \mathcal{F}$}.  The function $e_n\in [t]^n$ defined by $e_n(x)= 1+ \mathbbm{1}_{\{x=n\}}$ does not lie in $\mathcal{F}$. By saturation, there is a set $Q\subseteq \mathcal{F}$ of size $\vert P\vert -1\geq 2$ such that $Q\cup\{e_n\}$ induces a copy of $P$. As $e_n(n)>f(n)$ for all $f\in Q$, $e_n$ must be a $\leq$-maximal element in $Q\cup\{e_n\}$. On the other hand, there can be at most one element of $\mathcal{F}$ (and thus $Q$) that could be lying below it, namely the constant function $\mathbf{1}$, if it belongs to $\mathcal{F}$.

				Since $Q$ contains at least two elements, it follows that there is at least one element of $Q$ that is incomparable with $e_n$. This in turn implies that the poset $P$ induced by $Q\cup\{e_n\}$ contains at least two distinct $\leq$-maximal elements. But then, consider the function $\mathbf{t}: \ x\mapsto t$. This function does not lie in $\mathcal{F}$ (since $\mathbf{t}(n)\neq1$), but it lies above every element of $\mathcal{F}$, and so there is no copy of $P$ in $\mathcal{F}\cup\{\mathbf{t}\}$ (since $\mathcal{F}$ is $P$-free and since $P$ has no unique $\leq$-maximal element), a contradiction. 
				
				\noindent \textbf{Case 2: $\exists f\in \mathcal{F}$ with $f(n)>1$.} Let $f$ be $\leq$-minimal in $\mathcal{F}$ subject to the condition $f(n)>1$. Let $f^-\in[t]^n$ be the function defined by $f^-(x):= f(x) -\mathbbm{1}_{\{x=n\}}$. Since $n$ is not separating for $\mathcal{F}$, $f^- \notin\mathcal{F}$. By saturation, there exists $Q\subseteq \mathcal{F}$ such that $Q\cup\{f^-\}$ induces a copy of $P$.

				If $f\in Q$, then $f$ covers $f^-$ in $Q\cup \{f^-\}$ (since it does so in $[t]^n$). By UCTP, this implies there exists $g\in Q\subseteq \mathcal{F}$ such that $f$ covers $g$. Now $g\leq_{[n-1]} f =_{[n-1]}f^-$, so we must have $g(n)>f^-(n)=f(n)-1$. Since $f$ covers $g$, it follows that $g(n)=f(n)>1$, contradicting the $\leq$-minimality of $f$.

				On the other hand if $f\notin Q$, then we claim that $Q\cup\{f\}$ induces a copy of $P$ in $\mathcal{F}$. Indeed, note first that for any $g\in Q$ we have that $g\leq f^-$ implies $g\leq f$. Secondly if $g\geq f^-$ then either $g(n)=f^-(n)=f(n)-1$, in which case $f, g$ show $n$ is separating for $\mathcal{F}$, a contradiction, or $g(n)\geq f(n)$, in which case $g\geq f$. Finally if $g,f^-$ are incomparable, then so are $g,f$ unless $g\leq_{[n-1]} f^-=_{[n-1]}f$ and $g(n)=f^{-}(n)+1=f(n)$, but the latter would contradict the $\leq$-minimality of $f$. Thus it follows that we can replace $f^-$ by $f$ in $Q\cup\{f^-\}$ and preserve all poset relations, so that $Q\cup\{f\}\subseteq \mathcal{F}$ induces a copy of $P$ as claimed, contradicting the fact that $\mathcal{F}$ is induced $P$-free.
			\end{proof}
			

			\begin{lemma}\label{lemma: saturation for two incomparable elements}
				A family $\mathcal{F}\subseteq [t]^n$ is induced $A_2$-free saturated if and only if it is a maximal chain; in particular, $\mathrm{sat}^{\star}([t]^n, A_2)=nt+1$.
			\end{lemma}
			\begin{proof}
				Any induced $A_2$-free saturated family in $[t]^n$ must be a chain, and any non-maximal chain in $[t]^n$ can be extended to a longer chain, and hence fails to be $A_2$-saturated. On the other hand, any maximal chain in $[t]^n$ is easily seen to be $A_2$-saturated and contains $nt+1$ elements exactly. The lemma follows.
			\end{proof}
			
			\begin{proof}[Proof of Theorem~\ref{theorem: strong dichotomy UCTP}]
				The unique poset on two elements with the UCTP property is $A_2$, for which the conclusion of Theorem~\ref{theorem: strong dichotomy UCTP} holds by Lemma~\ref{lemma: saturation for two incomparable elements}. On the other hand, if $P$ is a UCTP poset on at least $3$ elements then by Lemma~\ref{lemma: UCTP implies separating} any induced $P$-free saturated family $\mathcal{F}\subseteq [t]^n$ is separating. It then follows from Lemma~\ref{lemma: separating implies large} that $\vert\mathcal{F}\vert\geq \Omega(\sqrt{n})$. The theorem follows.
			\end{proof}
			
			\subsection{Strong dichotomy: proof of Theorem~\ref{theorem: strong dichotomy}}
			\begin{lemma}\label{lemma: non-separation and only extreme values imply liftable}
				Suppose $\mathcal{F}$ is an induced $P$-free saturated family in $[t]^n$ and $i\in [n]$ is such that
				\begin{enumerate}[(a)]
					\item $i$ is not separating for $\mathcal{F}$;
					\item $f(i)\in\{1,t\}$ for all $f\in \mathcal{F}$.
				\end{enumerate}
				Then $L_i(\mathcal{F})$ is an induced $P$-free saturated family in $[t]^{n+1}$.
			\end{lemma}
			\begin{proof}
				This is proved in a somewhat similar fashion to~\cite[Lemma 2.3]{FPST}. Assume without loss of generality that $i=n$. As the operator $L_n$ preserves comparability relations, $\mathcal{F}$ is induced $P$-free and we only need to establish saturation. Suppose $g\in [t]^{n+1}\setminus L_n(\mathcal{F})$.

				If $g(n)=g(n+1)$, then it must be the case that $D_{n+1}(g)\notin \mathcal{F}$ and hence that there exists $Q\subseteq \mathcal{F}$ such that $Q\cup\{D_{n+1}(g)\}$ induces a copy of $P$ in $[t]^{n}$. It is immediate that $L_n(Q)\subseteq L_n(\mathcal{F})$ together with $g= L_n\circ D_{n+1}(g)$ induces a copy of $P$ in $[t]^{n+1}$, and we are done in this case. We may therefore assume by symmetry that $g(n)<  g(n+1)$. Since $n$ is non-separating, at least one of $D_{n+1}(g), D_n(g)$ is not in $\mathcal{F}$.

				Suppose first of all that $D_{n+1}(g)\notin \mathcal{F}$. Then there exists $Q\subseteq \mathcal{F}$ such that $Q\cup\{D_{n+1}(g)\}$ induces a copy of $P$ in $[t]^n$. Clearly for any $h\in [t]^n$, if $h\leq D_{n+1}(g)$ then $L_n(h)\leq g$, and if $h, D_{n+1}(g)$ are incomparable then so are $L_n(h), g$. Furthermore, $D_{n+1}(g)< h$ implies $g\leq L_n(h)$ unless $g\leq_{[n-1]} h$ and $g(n)\leq h(n)< g(n+1)$. Thus $L_n(Q)\cup\{g\}$ induces a copy of $P$ in $[t]^ {n+1}$ unless there exists some $h_1\in Q\subseteq \mathcal{F}$ with $g\leq_{[n-1]} h_1$ and $g(n)\leq h(n)< g(n+1)\leq t$. By our assumption (b) this means $h_1(n)=1$.
				
				If such an $h_1$ exists, then, as $n$ is non-separating for $\mathcal{F}$, it must be the case that $D_{n}(g)\notin \mathcal{F}$. This in turn implies there exists $Q'\subseteq \mathcal{F}$ such that $Q'\cup\{D_{n}(g)\}$ induces a copy of $P$ in $[t]^n$. By a symmetric argument to the one in the paragraph above, we have that $L_n(Q')\cup\{g\}$ induces a copy of $P$ in $[t]^{n+1}$ unless there exists $h_0 \in Q'\subseteq \mathcal{F}$ with $g\geq_{[n-1]} h_0$ and $g(n+1)\geq h_0(n)>g(n)\geq 1$. By our assumption (b), this means $h_0(n)=t$. Now $h_0,h_1 \in \mathcal{F}$ satisfy $h_0\leq_{[n-1]} g\leq_{[n-1]}h_1$ and $h_0(n)=t> 1= h_1(n)$  so $n$ must be separating for $\mathcal{F}$, a contradiction.

				The case $D_n(g)\notin \mathcal{F}$ follows 
				similarly, mutatis mutandis.
				
			\end{proof}
			
			\begin{lemma}\label{lemma: non separating implies can make 1k bounded}
				Suppose $\mathcal{F}$ is an induced $P$-free saturated family in $[t]^n$ and $i\in [n]$ is  not separating for $\mathcal{F}$. Then there exists an induced $P$-free saturated family $\mathcal{F}'$ in $[t]^n$ with size $\vert \mathcal{F}'\vert =\vert\mathcal{F}\vert$ such that $f(i)\in\{1,t\}$ for all $f\in \mathcal{F}$.
			\end{lemma}
			\begin{proof}
				Let $\mathcal{F}$ be an induced $P$-free saturated family in $[t]^n$, and assume some coordinate $i\in [n]$ is not separating for $\mathcal{F}$. Assume without loss of generality that $i=n$. Consider the set $\mathcal{F}(n):=\{f(n):\ f\in \mathcal{F}\}$. If $\mathcal{F}(n)\subseteq \{1,t\}$, we have nothing to show. Otherwise, there exists a minimal $j\in \{2, \ldots, t-1\}$ such that $j\in \mathcal{F}(n)$.

				Let $d=d^{(j)}$ be the operator sending $f\in [t]^n$ to the function $d(f)\in [t]^n$ defined by
				\begin{align*}
					d(f): \ x \mapsto \begin{cases}
						j-1 & \textrm{ if }x=n \textrm{ and }f(x)=j\\
						f(x) & \textrm{ otherwise.}
					\end{cases}    
				\end{align*}
				Further let $\ell$ be the operator sending $f\in [t]^n$ to the function $\ell(f) \in [t]^n$ defined by
				\begin{align*}
					\ell(f): \ x \mapsto \begin{cases}
						f(x)+1 & \textrm{ if }x=n \textrm{ and }f(x)<t\\
						f(x) & \textrm{ otherwise.}
					\end{cases}    
				\end{align*}
				\begin{claim}\label{claim: d's action on F}
					The operator $d$ acts injectively on $\mathcal{F}$ and preserves poset relations within it. In particular, $\vert d(\mathcal{F})\vert=\vert \mathcal{F}\vert$, $d(\mathcal{F})$ is induced $P$-free, and furthermore $n$ is not separating for $\mathcal{F}$.  
				\end{claim}
				\begin{proof}
					Since coordinate $n$ is not separating for $\mathcal{F}$, for any $f \in [t]^n$ we have that either $d(f)=f$ or at most one of $\{f, d(f)\}$ belongs to $\mathcal{F}$. This shows that $d$ acts injectively on $\mathcal{F}$.

					Trivially, $f\leq g$ implies $d(f)\leq d(g)$. Further, if $f$ and $g$ are incomparable, then so are $d(f)$ and $d(g)$ unless
					coordinate $n$ is separating for $\{f,g\}$, in which case at most one of $f$ and $g$ can lie in $\mathcal{F}$.  Thus the operator $d$ preserves poset relations inside $\mathcal{F}$. In particular, the family $d(\mathcal{F})\subseteq [t]^n$ is induced $P$-free.

					Furthermore $n$ is not separating for $d(\mathcal{F})$: indeed suppose that for some $f,g\in \mathcal{F}$ we have $d(f)\leq_{[n-1]}d(g)$ and $d(f)(n)>d(g)(n)$. By construction of $d(\mathcal{F})$, neither of $d(f)(n), d(g)(n)$ is equal to $j$. If $d(g)(n)=j-1$, then $g(n)\leq j < j+1\leq d(f)(n)=f(n)$, and otherwise $g(n)=d(g)(n) <d(f)(n)\leq f(n)$; in either case $n$ is separating for $\{f,g\}\subseteq \mathcal{F}$, a contradiction. 
				\end{proof}
				\noindent Our aim is now to show $d(\mathcal{F})$ is induced $P$-saturated; this is the content of the next two claims.
				\begin{claim}\label{claim: saturated in F minus dF}
					Suppose $g\in \mathcal{F}\setminus d(\mathcal{F})$. Then $d(\mathcal{F})\cup\{g\}$ contains a copy of $P$.
				\end{claim}
				
				\begin{proof}
					Since $g \in \mathcal{F}\setminus d(\mathcal{F})$ and $n$ is not separating for $\mathcal{F}$, we must have $g(n)=j$ and further $\ell(g)\notin \mathcal{F}$. Now, by saturation there exists $Q\subseteq \mathcal{F}$ such that $Q\cup\{\ell(g)\}$ induces a copy of $P$.

					It is not hard to see that $d(Q)\cup\{g\}$ induces a copy of $P$. Indeed, let $h\in Q$. First of all $d(h)\neq g$ (since by assumption $g\notin d(\mathcal{F})$). Secondly, if $h\geq \ell(g)$, then certainly $d(h)\geq g$. Thirdly, if $\ell(g)\geq h$ then we must have $g\geq_{[n-1]} h$; as $n$ is not separating for $\{g,h\}\subseteq \mathcal{F}$, this implies that $h(n)\leq j$ and hence that $g\geq d(h)$ as desired. Finally, if $h, \ell(g)$ are incomparable, then the only way $d(h)$ and $g$ could be comparable (given the fact $\ell(g)(n)=g(n)+1$ and $h(n)\geq d( h(n))\geq h(n)-1$) would be if $g\leq_{[n-1]}h$ and $g(n)=j \leq d(h)(n)=h(n)$; however $d(h)(n)\geq j$ would imply $h(n)=d(h)(n)\geq j+1$ and thus $\ell(g) \leq h$, a contradiction. This proves our claim. 
				\end{proof}
				\begin{claim}\label{claim: saturated in the rest}
					Suppose $g\in [t]^n\setminus\left(\mathcal{F}\cup d(\mathcal{F})\right)$. Then $d(\mathcal{F})\cup\{g\}$ contains a copy of $P$.
				\end{claim}
				\begin{proof}
					We must find a set $Q\subseteq d(\mathcal{F})$ such that $Q\cup\{g\}$ induces a copy of $P$. By saturation, we know there exists $Q'\subseteq \mathcal{F}$ such that $Q'\cup\{g\}$ induces a copy of $P$.  
					Consider first of all the set $d(Q')\cup \{g\}$. If this induces a copy of $P$, then we have found our desired set $Q=d(Q')\subseteq d(\mathcal{F})$, and we are done. Otherwise, since $d$ preserves poset relations within $Q$ by Claim~\ref{claim: d's action on F}, it must be the case that $d$ fails to preserve a poset relation between an element of $Q'$ and $g$. 
					Since $g\geq f$ implies $g\geq d(f)$, it follows that one of the following must occur:
					\begin{enumerate}[(a)]
						\item $g(n)=j$ and $\exists f_1\in Q'$: $g<_{[n-1]} f_1$ with $f_1(n)=j$ (so that $g\leq f_1$ but $g\not\leq d(f_1)$);
						\item $g(n)=j-1$ and $\exists f_0\in Q'$: $f_0\leq_{[n-1]} g$ with $f_0(n)=j$ (so that $f_0\not\sim g$ but $d(f_0)\leq g$).
					\end{enumerate}
					\noindent \textbf{Case (a):} By separation, we have that $\ell(g)\notin \mathcal{F}$ (since otherwise the coordinate $n$ would be separating for the pair of distinct functions $\{\ell(g), f_1\}\subseteq \mathcal{F}$). Note that $j+1\leq t$, whence $\ell(g)\in [t]^n$. Thus by saturation there exists $Q''\subseteq \mathcal{F}$ such that $Q''\cup \{\ell(g)\}$ induces a copy of $P$.

					Set $Q:=d(Q'')$, and consider a function $f\in Q''$. It is easy to show that the poset relations inside the ordered pairs $(f,\ell(g))$ and $(d(f),g)$ are identical. 
					Indeed, if $f\geq \ell(g)$ then certainly $d(f)\geq g$. Further if $f\leq \ell(g)$ then $d(f)\leq g$ unless $f\leq_{[n-1]}g$ and $\ell(g)(n)=f(n)=j+1$ both held,
					but then the coordinate $n$ would be separating for the pair of distinct functions $\{f,f_1\}$, a contradiction. Finally, if $f\not\sim \ell(g)$, then $d(f)\not\sim g$ unless $g\leq_{[n-1]}f$ and $d(f)(n)=g(n)=j$, which is impossible by definition of the operator $d$.

					It follows from the analysis in the paragraph above (and the fact that by Claim~\ref{claim: d's action on F} $d$ preserves poset relations when applied to $Q''$) that the map sending $f\in Q''$ to $d(f) \in Q$ and $\ell(g)$ to $g$ is a poset embedding, and therefore $Q\cup\{g\}$ induces the desired copy of $P$.
					
					\noindent \textbf{Case (b):} Since $g\notin d(\mathcal{F})$ and $g(n)=j-1$, it must be the case that $\ell(g)\notin \mathcal{F}$. Thus by saturation there exists $Q''\subseteq \mathcal{F}$ such that $Q''\cup\{\ell(g)\}$ induces a copy of $P$.

					Set $Q=d(Q'')$, and consider a function $f\in Q''$. Again, it is easily shown that the ordered pairs $(f,\ell(g))$ and $(d(f),g)$ satisfy the same poset relations. Indeed, since $g=d(\ell(g))$, it is trivial that $f\leq \ell(g)$ implies $d(f)\leq g$, and $f\geq \ell(g)$ implies $d(f)\geq g$. Further, if $f\not\sim \ell(g)$, then $d(f)\not\sim g$ unless $g \leq_{[n-1]}f$ and $f(n)=j-1$ both held, but then the coordinate $n$ would be separating for the pair of distinct functions $\{f,f_0\}$, a contradiction. It thus follows as in the previous case that $Q\cup\{g\}$ induces a copy of $P$ as desired.

					Our claim follows.
				\end{proof}
				Taken together, Claims~\ref{claim: saturated in F minus dF} and~\ref{claim: saturated in the rest} show that $d^{(j)}(\mathcal{F})=d(\mathcal{F})$ is an induced $P$-free saturated family $[t]^n$ with the same size as $\mathcal{F}$ for which the coordinate $n$ is not separating and which contains no function $f$ with $f(n)\in\{2,\ldots, j-2\}\cup\{j\}$. 
				
				Iterating our argument, we find that $d^{(\leq j)}\mathcal{F}:= d^{(2)}\circ d^{(3)}\circ \cdots \circ d^{(j)} (\mathcal{F})$ is an induced $P$-free saturated family $[t]^n$ with the same size as $\mathcal{F}$ for which the coordinate $n$ is not separating and which contains no function $f$ with $f(n)\in\{2,\ldots, j-2, j-1, j\}$.

				Iterating this argument $t-1-j$ times, we find that $\mathcal{F}':= d^{(\leq t-1)}\circ d^{(\leq t-2)}\circ \cdots \circ d^{(\leq j)} (\mathcal{F})$ is a family satisfying the conclusion of Lemma~\ref{lemma: non separating implies can make 1k bounded}.
			\end{proof}
			
			\begin{proof}[Proof of Theorem~\ref{theorem: strong dichotomy}]
				Taken together, Lemmas~\ref{lemma: non-separation and only extreme values imply liftable} and~\ref{lemma: non separating implies can make 1k bounded} imply that if for any $N$ there exists an induced $P$-free saturated family $\mathcal{F}\subseteq [t]^N$ for which some coordinate $i\in[N]$ is not separating, then for all $n\geq N$ we can construct an induced $P$-free saturated family in $[t]^n$ of size at most $\vert \mathcal{F}\vert$. Thus in this case we have $\mathrm{sat}^{\star}([t]^n, P)=O(1)$.

				On the other hand, if for all $n$ and all induced $P$-free saturated families $\mathcal{F}\subseteq [t]^n$ we have that every coordinate $i\in [n]$ is separating for $\mathcal{F}$, then it follows from Lemma~\ref{lemma: separating implies large} that $\mathrm{sat}^{\star}([t]^n, P)$) is at least $\Omega(\sqrt{n})$. The theorem follows.
			\end{proof}

			\subsection{Monotonicity: proof of Theorem~\ref{theorem: monotonicity}} 
		\begin{lemma}\label{lemma: altar}
			Let $P$ be a poset and $\mathcal{F}$ be an induced $P$-free saturated family in $[t]^n$. Suppose $i\in [n]$ is not separating for $\mathcal{F}$. Then $D_i(\mathcal{F})$ is an induced $P$-free saturated family in $[t]^{n-1}$.
		\end{lemma}
		\begin{proof}
			Assume without loss of generality that $i=n$. Consider $g\in [t]^{n-1}\setminus D_{n}(\mathcal{F})$. Then for every $j\in[t]$, the function $g^{j}$ sending $x$ to $g(x)$ if $x\in [n-1]$ and to $j$ if $x=n$ belongs to $[t]^{n}\setminus \mathcal{F}$. By saturation, there exists some set $Q^j\subseteq \mathcal{F}$ such that $Q^j\cup\{g^j\}$ induces a copy of $P$. Since $n$ is not separating for $\mathcal{F}$, the restriction of $D_n$ to $\mathcal{F}$ is injective (indeed if $D_n(f)=D_n(f')$ for some $f\neq f'$, then the coordinate $n$ would be separating for $\{f,f'\}$), and the operator $D_n$ preserves incomparability inside $Q^j$.
			What is more, $D_n$ preserves all comparability relations. Thus $D_n(Q_j)\cup\{g\}$ will induce a copy of $P$ unless there exists $h^j\in Q^j$ such that $h^j, g^j$ are incomparable elements of $[t]^{n}$, but $D_n(h^j),g$ are comparable in $[t]^{n-1}$. Assuming for a contradiction that this occurs for every $j\in [t]$, we can then define
			\begin{align*}
				s(j):=\begin{cases}
					1 & \textrm{ if } D_n(h^j)\geq g\\
					0 & \textrm{ if } D_n(h^j) \leq g.
				\end{cases}
			\end{align*}
			
			Note that $s(j)=1$ implies $h^j(n)\leq j-1$, while $s(j)=0$ implies $h^j(n)\geq j+1$. In particular we must have $s(t)=1$ and $s(1)=0$. It follows that there exists some index $j^{\star}\in [t-1]$ such that $s(j^{\star})=0$ and $s(j^{\star}+1)=1$. This in turn means that $D_n(h^{j^{\star}+1})\geq g \geq D_n(h^{j^{\star}})$ and $h^{j^{\star}+1}(n)\leq j^{\star}<j^{\star}+1\leq h^{j^{\star}}(n)$. The functions $h^{j^{\star}}, h^{j^{\star}+1}$ thus show that coordinate $n$ is separating for $\mathcal{F}$, contradicting our assumption.
		\end{proof}
		\begin{proof}[Proof of Theorem~\ref{theorem: monotonicity}]
			Suppose $\mathrm{sat}^{\star}([t]^n, P)=O(1)$ and set $C:=\limsup_{n\rightarrow \infty}\mathrm{sat}^{\star}([t]^n, P)$. By Lemma~\ref{lemma: separating implies large}, there exists $N\in \mathbb{Z}_{>0}$ such that for all $n\geq N$ and every induced $P$-free saturated family $\mathcal{F}\subseteq [t]^n$ of minimum size, there is at least one coordinate $i\in [n]$ which is not separating for $\mathcal{F}$.

			By Lemma~\ref{lemma: altar}, it follows that $D_i(\mathcal{F})$ is an induced $P$-free saturated family in $[t]^{n-1}$, and hence that $\mathrm{sat}^{\star}([t]^{n-1}, P)\leq \mathrm{sat}^{\star}([t]^{n}, P)$. On the other hand, it follows from Lemmas~\ref{lemma: non separating implies can make 1k bounded} and~\ref{lemma: non-separation and only extreme values imply liftable} combined that for all $n\geq N$ we have $\mathrm{sat}^{\star}([t]^n,P)\geq \mathrm{sat}^{\star}([t]^{n+1},P)$. Taken together these inequalities imply that for all $n>N$ we have $\mathrm{sat}^{\star}([t]^n,P)=C$.
		\end{proof}


			\subsection{Polynomial upper bounds: proof of Theorem~\ref{theorem: polynomial upper bound}}
			\noindent In this section we shall prove polynomial upper bounds for the induced saturation function of posets $P$ that can be written as a poset sum $P=P_1\star A_k\star P_2$ for some $k\geq 2$ and (possibly empty) posets $P_1$ and $P_2$ (recall the poset sum $P\star Q$ of two posets $P,Q$ was defined in Section~\ref{section: results}). For $r\in [tn]$, the \emph{$r$-th layer} of the hypergrid $[t]^n$, denoted by $([t]^{n})^{(r)}$ is the collection of $f\in [t]^n$ with $\sum_{i\in [n]}f(i)=r$. 
			\begin{definition}
				Given a poset $P$, we let $h^{+}(P)$ (respectively $h^-(P)$) denote the least $h$ such that for some $n$, there exists an induced copy of $P$ inside the top $r$ (respectively bottom $r$) layers of the hypercube $[2]^{n}$. 
				Further, we let $w^+(P)$ denote the least $n$ for which the top $h^+(P)$ layers of $[2]^{n}$ contain an induced copy of $P$, and we define $w^-(P)$ mutatis mutandis. 
			\end{definition} 
			\begin{remark}\label{remark: height well defined}
				The functions $h^+$ and $h^-$ are well-defined: given an enumeration of the elements of $P$ as $v_1, \ldots, v_{\vert P\vert}$, the map $v_i \mapsto \{j\in[\vert P\vert]: \  v_j\leq_P v_i\}$ is an embedding of $P$ in the hypercube poset $[2]^{\vert P\vert}$, hence $h^+(P), h^-(P)\leq \vert P\vert +1$.
			\end{remark}
			\begin{remark}[Embedding of the hypergrid in the hypercube]\label{remark: embedding in the hypercube}
				If $P$ can be embedded in the top $r$ layers of $[t]^{n}$, then it can also be embedded in the top $r$ layers of $[2]^{(t-1)n}$. Indeed, one can map $f\in [t]^n$ to a function $E(f)\in [2]^{(t-1)n}$ defined as follows: for $i\in [n]$ and $j\in [t-1]$, set
				\begin{align*}
					E(f): \ (t-1)(i-1)+j \ \mapsto \begin{cases} 1 & \textrm{ if  }\ f(i)\leq j\\
						2 & \textrm{ otherwise.}
					\end{cases}
				\end{align*}
				The map $E$ is easily seen to be an injective embedding of $[t]^n$ in $[2]^{(t-1)n}$ that preserves all poset relations. 
			\end{remark}
			
			Our proof of Theorem~\ref{theorem: polynomial upper bound} is inspired by the elegant VC-dimension argument of Bastide, Groenland, Ivan and Johnston~\cite{BGIJ}. Unfortunately this argument fails in the hypergrid setting (see the third point in the concluding remarks). With a bit of work, however, we are able to obtain a weaker result than the hypercube result of~\cite{BGIJ} in the hypergrid setting, using a generalisation of VC dimension known as \emph{Natarajan dimension}.
			\begin{definition}\cite{Natarajan}
				The Natarajan dimension $\mathrm{dim}_N(\mathcal{F})$ of a family $\mathcal{F}\subseteq [t]^n$ is the greatest $m$ such that there exists a subset $X\subseteq[n]$ of size $m$ and functions $F^-, F^+ \in [t]^n$ such that:
				\begin{enumerate}[(a)]
					\item $F^-(x)<F^+(x)$ for all $x\in X$
					\item all $S\subseteq X$ there exists $f_S\in \mathcal{F}$ with 
					$f_S(x)=\begin{cases}
						F^+(x) & \textrm{ if }\ x\in S,\\
						F^-(x) & \textrm{ if }\ x\in X\setminus S.
					\end{cases}$
				\end{enumerate}
				\end{definition}
				%
				\noindent The following generalisation of the Sauer--Shelah lemma to Natarajan dimension, similar in spirit to an earlier result of Natarajan~\cite{Natarajan}, was proved by Haussler and Long in~\cite{HL}.
				\begin{proposition}\cite[Corollary 5]{HL}\label{prop: natarajan}
					Let $d,t,n \in \mathbb{Z}_{>0}$. Let $\mathcal{F}\subseteq [t]^n$ satisfy $\mathrm{dim}_N(\mathcal{F})\leq d$. Then 
					\begin{align*}
						\vert\mathcal{F}\vert \leq \sum_{i=0}^d \binom{n}{i}{\binom{t}{2}}^i
					\end{align*}
					In particular, for all $n$ sufficiently large $\vert \mathcal{F}\vert \leq (nt^2)^d$.
				\end{proposition}
				
				\begin{proof}[Proof of Theorem~\ref{theorem: polynomial upper bound}]
					Given $P=P_1\star A_k\star P_2$, let $\mathcal{F}_0\subseteq [t]^n$ consist of the union of the top $h^+(P_1)$ and the bottom $h^-(P_2)$ layers of $[t]^n$. By Remark~\ref{remark: embedding in the hypercube}, this is easily seen to be a $P_1\star A_1\star P_2$-free poset. Greedily add sets to $\mathcal{F}_0$ until an induced $P_1\star A_k\star P_2$-free saturated family $\mathcal{F}$ is obtained. Set $N= w^+(P_1)+w^-(P_2)+ 2\lceil\log_2(k)\rceil$, and note this is a constant depending only on $P$. 
					\begin{claim}\label{claim: embedding}
						$[2]^N$
						contains a copy of $P_1\star A_k\star P_2$ such that the subset corresponding to $P_1$ lies inside the top $h^+$ layers and the subset corresponding to $P_2$ lies inside the bottom $h^-$ layers. 
					\end{claim}
					\begin{proof}
						Indeed, suppose $v\mapsto f_v$ is a poset embedding of $P_1$ in the top $h^+(P_1)$ layers of  $[2]^{w^+(P_1)}$. Then setting $\widetilde{f_v}(x)$ to be $f_v(x)$ if $x\in [w^+(P_1)]$ and $2$ otherwise, we have that $v\mapsto \widetilde{f_v}$ is a poset embedding of $P_1$ in the top $h^+(P_1)$ layers of $[2]^{N}$.

						Next, suppose $v\mapsto g_v$ is a poset embedding of $P_2$ in the bottom $h^-(P_2)$ layers of  $[2]^{w^-(P_2)}$. Then setting $\widetilde{g_v}(x)$ to be $g_v(x)$ if $x\in [w^+(P_1)+w^-(P_2)]\setminus [w^+(P_1)]$ and $1$ 
						otherwise, we have that $v\mapsto \widetilde{g_v}$ is a poset embedding of $P_2$ in the bottom $h^-(P_2)$ layers of $[2]^{N}$ that lies strictly below our embedding of $P_1$.

						Finally, arbitrarily select an antichain $\{a_1, \ldots, a_k\}$ inside $[2]^{2\lceil\log_2(k)\rceil}$. (This is always possible since the middle layer of that hypercube is an antichain of size $\binom{2\lceil\log_2(k)\rceil}{\lceil \log_2k\rceil}\geq k$ for all $k\geq 1$.) 
						%
						%
						For $1\leq i \leq k$, define a function $h_i\in [2]^N$ by setting
						\begin{align*}
							h_i: \ {x} \longmapsto \begin{cases}
								2 & \textrm{if either of $x\in [w^+(P_1)+w^-(P_2)]\setminus [w^+(P_1)]$ or $x- \left(w^+(P_1)+w^-(P_2)\right)\in a_i$ holds}\\
								1 & \textrm{otherwise.}
							\end{cases} 
						\end{align*}
						Mapping the $i$-th element from $A_k$ to $h_i$ gives a poset embedding of $A_k$ in $[2]^{N}$ that lies strictly below our poset embedding of $P_1$ and strictly above our poset embedding of $P_2$, proving our claim.
					\end{proof}
					
					Suppose for a contradiction that $\mathrm{dim}_N(\mathcal{F})\geq N$. 
					Without loss of generality we then have that $\exists F^-, F^+ \in [t]^n$ with $F^-(x)< F^+(x)$ for all $x\in [N]$,
					and further that $\forall S\subseteq [N]$, $\exists f_S\in \mathcal{F}$ with $f_S(x)=F^+(x)$ for all $x\in S$ and $f_S(x)=F^-(x)$ for all $x\in [N]\setminus S$.
					
					For all $S$ with $\vert S\vert \leq h^{-}(P_2)-1$, the function $x\mapsto 1+\mathbbm{1}_{\{x\in S\}}$ belongs to the bottom $h^-(P_2)$ layers of $[t]^{n}$ and hence to $\mathcal{F}$. Similarly, for all $S$ with $\vert S\vert \geq n- h^{+}(P_1)+1$, the function $x\mapsto k-1+\mathbbm{1}_{\{x\in S\}}$ belongs to the top $h^+(P_1)$ layers of $[t]^{n}$ and hence to $\mathcal{F}$. Consider now a poset embedding $\phi: \ P\rightarrow [2]^N$, which exists by Claim~\ref{claim: embedding}. Then we can define a poset embedding $\phi': \ P\rightarrow \mathcal{F}$ as follows. If $\phi(x)=S$, then let 
					\begin{align*}
						\phi'(x)= \begin{cases}
							1 + \mathbbm{1}_{\{x\in S\}} & \textrm{ if } \vert S\vert \leq h^-(P_2)-1\\
							t-1 + \mathbbm{1}_{\{x\in S\}} & \textrm{ if } \vert S\vert \geq n-h^+(P_1)+1\\
							f_S & \textrm{ otherwise.}
						\end{cases}
					\end{align*}
					Since $f^-(x)<f^+(x)$ for all $x\in [N]$, we have $1\leq f^-(x)<t$ and $1<f^+(x)\leq t$ for all such $x$, and $\phi'$ is easily seen to be a poset embedding of $P$ in $\mathcal{F}$, a contradiction.

					It follows that $\mathrm{dim}_N(\mathcal{F})< N$, and hence by Proposition~\ref{prop: natarajan} that $\vert \mathcal{F}\vert \leq (nt^2)^{N-1}$ for all $n$ sufficiently large. As we observed above, $c=N-1$ is a constant depending only on $P$. This concludes the proof of the theorem.
					%
				\end{proof}

				\subsection{Bounds for chains and antichains: proof of Proposition~\ref{proposition: small examples}}
				\begin{proof}[Proof of Proposition~\ref{proposition: small examples}]
					\noindent\textbf{For the chain $C_k$}, we give a construction similar to that given in~\cite{GKLPPP} for the hypercube. Consider the family $\mathcal{F}_k$ of all functions $f\in [t]^n$ such that either $f([k-3])\subseteq [2]$ and $f([n]\setminus[k-3])= \{1\}$ or $f([k-3])\subseteq \{1,t\}$ and $f([n]\setminus[k-3])= \{t\}$.

					Given $f\in \mathcal{F}_k$, let $p(f)\in [2]^{k-3}$ 
					be the function mapping $x$ to $1$ if $f(x)=1$ and $2$ otherwise. 
					For $f,f'\in \mathcal{F}_k$, note that $p(f)=p(f')$ if and only if inside $[k-3]$ $f$ and $f'$ take the value $1$ on exactly the same subset, while outside $[k-3]$ one of $f$, $f'$ only takes the value $1$ and the other only takes the value $t$. 
					%
					%
					%
					It is easily checked that if $C\subseteq \mathcal{F}_k$ induces a chain of length $\ell$, then $p(C)$ is a chain of length at least $\ell-1$ in $[2]^{k-3}$. It follows that the longest chain in $\mathcal{F}_k$ has length at most $k-2+1=k-1$, and hence that $\mathcal{F}_k$ is (induced) $C_k$-free. 

					Consider then a function $f\in [t]^n\setminus \mathcal{F}_k$. Set $S:=\{x \in [k-3]: \ f(x)>1\}$. Assume without loss of generality that $S=[i]$ for some $i\leq k-3$. Then for all $j$: $0\leq j\leq i$, let $f_{j}^-$ be the function in $\mathcal{F}_k$ which is equal to $1$ on $[n]\setminus [j]$ and equal to $2$ on $[j]$. Further, for all $j$: $i\leq j \leq k-3$, 
					let $f^+_{j}$ 
					be the function in $\mathcal{F}_k$ which is equal to $1$ on $[k-3]\setminus[j]$ and equal to $t$ on $[n]\setminus ([k-3]\setminus[j])$. It is easily checked that $\{f^-_{j}: \ 0\leq j\leq i\}\cup \{f\}\cup\{f\}\cup\{f^+_{j}: \ i\leq j \leq k-3\}$ induces a copy of $C_k$ in $\mathcal{F}_k\cup\{f\}$. 
					Thus $\mathcal{F}_k$ is induced $P$-free saturated, and $\mathrm{sat}([t]^n, C_k)\leq \vert \mathcal{F}_k\vert = 2^{k-2}$, as claimed.

					\noindent\textbf{For the antichain $A_k$}, it is trivial to see that $(t-1)n+1\leq \mathrm{sat}^{\star}([t]^n, P) \leq (k-1)(t-1)n -k+3$. 
					Indeed let $\mathcal{F}$ be induced $A_k$-free saturated in $[t]^n$ for some $n>k$. By Dilworth's theorem, $\mathcal{F}$ can be decomposed into the union of at most $k-1$ disjoint chains. Arbitrarily extend these $k-1$ disjoint chains to maximal length chains in $[t]^n$, labelled as $C^1 , C^2, \ldots, C^{k-1}$. Now $\mathcal{F}\subseteq \bigcup_{j=1}^{k-1} C^j$, $\mathcal{F}$ is $A_k$-saturated and $\bigcup_{j=1}^{k-1} C^j$ is $A_k$-free, so it follows that  $\mathcal{F}=\bigcup_{j=1}^{t-1} C^j$. In particular $\mathcal{F}$ contains at least one maximal chain of $[t]^n$ and hence must have size at least $(t-1)n+1$, and it can have size at most $(k-1)((t-1)n+1)- 2(k-2)= (k-1)(t-1)n -k +3$.
				\end{proof}

	\section{Concluding remarks}\label{section: concluding remarks}

        Conjecture~\ref{conjecture: dichotomy} remains widely open, even in the hypercube $t=2$ setting. A major challenge is that we currently lack any effective combinatorial criterion for distinguishing posets $P$ with bounded induced saturation function from those for which $\mathrm{sat}^{\star}([2]^n, P)$ grows polynomially. 

        Moreover, the results and conjectures of Keszegh, Lemons, Martin, Patk\'os and P\'alv\"ogyi~\cite{KLMPP} and Ji, Patk\'os and Yue~\cite{JiPatkosYue25} on disjoint unions of mutually incomparable chains provide compelling evidence that such a criterion, if it exists, is likely to be highly non-trivial. Notably,~\cite[Conjecture 1.2]{JiPatkosYue25} states that the disjoint union of chains, not all of the same size, must necessarily have bounded induced saturation function, while based on results for small $k$ the authors of~\cite{KLMPP} speculated that the union of $k$ mutually incomparable copies of $C_2$ has bounded induced saturation number if and only if $k$ is odd. This suggests that the induced saturation function is extremely sensitive to even small changes in the structure of the poset.

        Turning to lower bounds in Conjecture~\ref{conjecture: dichotomy}, separation alone cannot yield bounds stronger than $\Omega(\sqrt{n})$, as shown by an example of Freschi, Piga, Sharifzadeh and Treglown~\cite[Example 3.1]{FPST}. Existing proofs of $\Omega(n)$ lower bounds instead tend to rely on additional structural properties of the poset $P$.

        For the upper bound, extending the polynomial bounds from~\cite{BGIJ} to the hypergrid setting presents significant obstacles. In particular, the analogue of the Sauer--Shelah lemma fails in this setting: the family of all functions $f\in [t]^n$ with $f(x)\neq j$ provides an example of a family $\mathcal{F}$ of size $(t-1)^n$ such that $\mathcal{F}_{\vert X} \neq [t]^X$ for every non-empty subset $X\subseteq [n]$. Using the Natarajan dimension also presents difficulties. In this setting, the functions $F^-$ and $F^+$ could be constant functions, for instance $\mathbf{r-1}$ and $\mathbf{r}$, which prevents one from projecting the bottom-most $h^-(P)-1$ layers in a copy of $P$ from $[t]^X$ into $[t]^n$ in a way that preserves incomparability with the top layer, unlike in the hypercube setting considered in~\cite{BGIJ}.

        Our elementary bounds on chains and antichains (Proposition~\ref{proposition: small examples}) should be improvable, though the transition from the hypercube to the general hypergrid setting introduces additional challenges. For instance, for induced antichain saturation, the arguments of Bastide, Groenland, Jacob and Johnston~\cite{BGJJ} rely on results of Lehman and Ron~\cite{LR} that do not extend directly to the hypergrid.

        It is also natural to ask whether there exists a poset $P$ and an integer $T\in \mathbb{Z}_{\geq 3}$ such that $\mathrm{sat}^{\star}([2]^n, P)=O(1)$ but $\mathrm{sat}^{\star}([T]^n, P)=\Omega(\sqrt{n})$. In other words, is having a bounded induced saturation function an intrinsic property of $P$, independent of the side-length $t$ of the hypergrid? At present, we do not know the answer to this question, and we do not have sufficient evidence to suggest a conjecture.

        In a related direction, suppose $\mathrm{sat}^{\star}([t]^n, P)=O(1)$. Then there is a least $N_0$ such that $[t]^{N_0}$ contains an induced copy of $P$, and, by Theorem~\ref{theorem: monotonicity}, a least $N_1$ such that $\mathrm{sat}^{\star}([t]^n, P)=\mathrm{sat}^{\star}([t]^{N_1}, P)$ for all $n \geq N_1$. It would be interesting to understand how large the gap $N_1-N_0$ can be, that is, how long it takes for $\mathrm{sat}^{\star}([t]^n, P)$ to stabilize.

        Finally, throughout this paper we have considered the regime where $t\geq 2$ is fixed and $n$ tends to infinity. In contrast, Gerbner, Nagy, Patk\'os and Vizer~\cite{GerbnerNagyPatkosVizer22} studied the regime where $n$ is fixed (in fact $n=2$) and $t$ tends to infinity. In both settings, a dichotomy phenomenon emerges: in our case the induced saturation function is either $O(1)$ or $\Omega(\sqrt{n})$, while in theirs it is either $O(1)$ or $\Omega(t)$. It would be natural to investigate the intermediate regime where $t$ grows with $n$, which is not covered by either of our works and appears to present a new class of problems.

						\bibliographystyle{plain} 
						\bibliography{bibliography}
					\end{document}